\documentclass[11pt]{article} 

\usepackage{amsmath, amsfonts, amssymb}
\usepackage{mathrsfs}
\usepackage{theorem}
\usepackage{bm}
\usepackage[dvipdf]{graphicx}

\RequirePackage[colorlinks,citecolor=blue,urlcolor=blue,bookmarksopen]{hyperref}

\newtheorem{mythm}{Theorem}[section]
\newtheorem{myprop}[mythm]{Proposition}
\newtheorem{mylem}[mythm]{Lemma}

\newtheorem{mydefn}[mythm]{Definition}

{\newtheorem{myrem}[mythm]{Remark}}
{
}%

\newcommand{\dis}{\displaystyle}

\def\R{\mathbb R}

\def\N{\mathbb N}
\def\C{\mathscr C}

\def\F{\mathscr F}

\def\E{\mathbb E}
\def\p{\mathbb P}
\def\e{\text{\rm{e}}}

\def\la{\langle}\def\d{\text{\rm{d}}}
\def\raa{\rangle}

\def\veps{\varepsilon}

\def\law{\mathscr{L}}

\def\pb{\mathscr{P}}

\def\wt{\widetilde}

\def\var{\mathrm{var}}
\def\W{\mathbb{W}}

\newenvironment{proof}{{\noindent\it Proof.}\ }{\hfill $\square$\par}

\numberwithin{equation}{section}

\allowdisplaybreaks

\begin{document}

\title{Existence of optimal feedback controls for McKean-Vlasov SDEs\footnote{Supported in part by National Key R\&D Program of China (No. 2022YFA1006000) and NNSFs of China (No. 12271397,  11831014)}}

\author{Jinghai Shao\thanks{Center for Applied Mathematics, Tianjin University, Tianjin 300072, China. Email: shaojh@tju.edu.cn
}}
\date{}
\maketitle

\begin{abstract}
  This work concerns the optimal control problem for McKean-Vlasov SDEs. We provide explicit conditions to ensure the existence of optimal Markovian feedback controls. Moreover, based on the flow property of the McKean-Vlasov SDE, the dynamic programming principle is established, which will enable to characterize the value function via the theory of Hamilton-Jacobi-Bellman equation on the Wasserstein space.
\end{abstract}



\section{Introduction}

The McKean-Vlasov stochastic differential equations (SDEs) were introduced to describe the asymptotic behavior of a generic element of a large population of particles with mean field interactions, and are closely related to a phenomenon referred usually to as propagation of chaos, which were extensively studied in the literature; see, for example, \cite{BFY15}, \cite{BLPR}, \cite{CDLL}, \cite{Wang18} and the monographs \cite{Car1, Car2} and references therein.

This paper investigates the optimal feedback control problem in finite horizon for McKean-Vlasov SDEs. Consider the following SDE:
\begin{equation}\label{o-1}
\d X_t=b(t,X_t,\law_{X_t},\alpha_t)\d t+\sigma(t,X_t,\law_{X_t})\d W_t,\quad X_s=\xi\ \text{with $\law_{\xi}=\mu$},
\end{equation}
where $b:[0,T]\times \R^d\times \pb(\R^d)\times \pb(U)\to \R^d$, $\sigma:[0,T]\times \R^d\times \pb(\R^d)\to \R^{d\times d}$; $(W_t)_{t\geq 0}$ is a $d$-dimension Wiener process; $\pb(\R^d)$ ($\pb(U)$) denotes the collection of all probability measures over $\R^d$ ($U$ respectively), called Wasserstein space for simplicity; $U$ is a compact set in $\R^k$ for some positive integer $k$; $\law_{\xi}$ denotes the distribution of the random variable $\xi$; $(\alpha_t)$   represents the control policy belong to the set of admissible controls $\Pi_{s,\mu}$, detailed in next section. Given certain measurable functions $f$ and $g$, we aim to minimize the following objective function: for $0\leq s<T<\infty$,
\begin{equation}\label{o-2} J(s,\mu;\alpha):=\E\Big[\int_s^Tf(r,X_r,\law_{X_r},\alpha_r)\d r+g(X_T,\law_{X_T})\Big].
\end{equation}
The corresponding value function is given by
\begin{equation}\label{o-3}
V(s,\mu)=\inf\nolimits_{\alpha\in \Pi_{s,\mu}} J(s,\mu;\alpha).
\end{equation}

The optimal control problem of McKean-Vlasov SDEs is motivated by the mean-field game theory developed by \cite{LL06} and \cite{HCM06}. It has been studied in \cite{AD10,BDL,CD15} by maximum principle method, in \cite{BFY15,BIMS,LP16,Pham17} by dynamic programming method.
As the state variable of the value function contains the probability measure, the optimal control problem \eqref{o-3} is essentially an infinite dimensional problem. To characterize the value function, many approaches have been proposed to develop  the viscosity solution theory to Hamilton-Jacobi-Bellman (HJB) equations on the Wasserstein space.  For instance, Burzoni et al. \cite{BIMS} used the linear functional derivative; Gangbo et al. \cite{GNT} used the Riemannian tangent space structure; Bensoussan et al. \cite{BFY15} studied the probability densities, which were viewed as elements in $L^2(\R^d)$ and develop the G\^ateaux differential structure in the Hilbert space $L^2(\R^d)$. Pham and Wei \cite{Pham17} adopted the approach of Lions' lifting  \cite{Card,Lions}, which provides  a lifting identification between measures and random variables, then used the theory on viscosity solutions to HJB equations in Hilbert space (cf. \cite{Li88}) to characterize the value function.  See \cite[Chapter 5]{Car1} for the discussion on the relationship between different derivatives in the Wasserstein space. The work \cite{BIMS} used the linear functional derivative on $\pb(\R^d)$, and the key point is the subtle construction of a distance-like function on the Wasserstein space. Besides, the set of admissible controls in \cite{BIMS} is restricted to contain only the measurable deterministic functions of time.

In this work our purpose is to study the existence of the optimal Markovian feedback controls.
There are very limited works on the existence of the optimal feedback controls for McKean-Vlasov SDEs. Under explicit conditions of the coefficients, we show the existence  by developing the compactification method for McKean-Vlasov SDEs. The compactification method has been developed in Kushner \cite{Ku72,Ku}, Haussmann and Lepeltier \cite{HL90}, Haussmann and Suo
\cite{Haus1} amongst others; see the recent survey Kushner \cite{Ku14} for more references. Although our idea is similar to \cite{Haus1}, the concrete technique is quite different. This can be seen from the fact that \cite{Haus1} cannot deal with the case that the cost function depends on the terminal value of the studied system (cf. \cite[Remark 2.2]{Haus1}), but our work can deal with the objective function containing the terminal cost. Then, using the flow property of McKean-Vlasov SDEs (cf. \cite{BLPR}), we establish the dynamic programming principle for $v(s,\mu)$ given by \eqref{o-3} as in \cite{BIMS}, \cite{Pham17}. Furthermore, we study the continuity of the value function. Based on the dynamic programming principle, one can proceed to characterize the value function as a (viscosity) solution to certain HJB equation on the Wasserstein space. In \cite{Sh23b}, we shall show that the value function is the unique viscosity solution to certain HJB equation in terms of Mortensen's derivative. In \cite{Sh23c}, we study the optimal control problem for McKean-Vlasov SDEs with reflection.

This work is organized as follows. In Section 2, we present the framework of the optimal control problem, especially introduce the set of control policies studied in this work.  Section 3 is devoted to the proof of the existing of the optimal feedback controls. In Section 4, we establish the dynamic programming principle, and investigate the continuity of the value function.

\section{Framework}

In a given probability space $(\Omega,\F,\p)$, we use $\law_X$ to denote the law of a random variable $X$. For two probability measures $\mu,\nu$ over $\R^d$, the total variation distance between them is defined by
\[\|\mu-\nu\|_\var=\sup\big\{|\la f,\mu \raa -\la f, \nu\raa|;\,f\in \mathscr{B}(\R^d), |f|\leq 1\big\},
\]where $\la f,\mu\raa=\int_{\R^d} f(x)\mu(\d x)$.
The $L^p$-Wasserstein distance between $\mu$ and $\nu$ is defined by
\begin{equation}\label{W-1}
\W_p(\mu,\nu)=\inf_{\Gamma\in \mathcal{C}(\mu,\nu)} \!\Big(\int_{\R^d\times\R^d}\!\!|x-y|^p\,\Gamma(\d x,\d y)\Big)^{\frac 1p},\quad p\geq 1,
\end{equation}
where $\mathcal{C}(\mu,\nu)$ denotes the collection of all couplings of $\mu$, $\nu$ on $\R^d\times\R^d$.
Let $T>0$ be a given constant throughout this work. $U$ is a compact subset of $\R^k$ for some $k\geq 1$. Let $\pb(U)$ and $\pb(\R^d)$ denote the set of probability measures over $U$ and  $\R^d$ respectively. For $p\geq 1$, put
\begin{align*}\pb_p(\R^d)&=\Big\{\mu\in \pb(\R^d);\, \int_{\R^d}|x|^p\mu(\d x)<\infty\Big\}.
\end{align*}
Let $\C([s,t];\R^d)$ be the path space of all continuous functions from $[s,t]$ to $\R^d$.

For a measurable function $u:U\to \R$, it can be generalized as a functional on $\pb(U)$ via
\[u(\alpha):=\int_{U}u(x)\alpha(\d x),\quad \forall \, \alpha\in \pb(U).\] In the following, we often use $u(\alpha)$ instead of $\la u,\alpha\raa$ or $\int_{U}u(x)\alpha(\d x)$ to simplify the notation without mentioning it again.

Consider the following controlled stochastic dynamical system characterized by a SDE of McKean-Vlasov type:
\begin{equation}\label{a-1}
\d X_t= b(t,X_t,\law_{X_t},\alpha_t )  \d t+\sigma(t,X_t,\law_{X_t})\d W_t,\quad X_s=\xi,
\end{equation}
where $b:[0,T]\times \R^d\times \pb(\R^d)\times \pb(U)\to \R^d$, $\sigma:[0,T]\times \R^d\times \pb(\R^d)\to \R^{d\times d}$, $(W_t)_{t\geq 0}$ is a $d$-dimensional Brownian motion; $(\alpha_t)$ is a $\pb(U)$-valued process represented the controls  imposed  on the studied system.

\begin{mydefn}\label{def-1}
Given $s\in [0,T)$, $\mu\in \pb_1(\R^d)$, a Markovian feedback control is a term $\Theta=(\Omega,\F,\p$, $ \{\F_t\}_{t\geq 0}, W_\cdot, X_\cdot, \alpha_\cdot)$ such that
\begin{enumerate}
  \item[$\mathrm{(i)}$] $(\Omega,\F,\p)$ is a probability space with filtration $\{\F_t\}_{t\geq 0}$;
  \item[$\mathrm{(ii)}$] $(W_t)_{t\geq 0}$ is an $\F_t$-adapted Brownian motion;
  \item[$\mathrm{(iii)}$] $(\alpha_t)_{t\in [s,T]}$ is a $\pb(U)$-valued measurable process on $\Omega$ adapted to $\{\F_t\}$  such that for any $\F_s$-measurable random variable $\xi$ satisfying $\law_\xi=\mu$, SDE \eqref{a-1} admits a weak solution $(X_t)_{t\in [s,T]}$  with initial value $X_s=\xi$, and the uniqueness in law holds for SDE \eqref{a-1}.
  \item[$\mathrm{(iv)}$] For each $t\in [s,T]$, there exists a measurable functional $F_t\R^d\to \pb(U)$ such that $\alpha_t=F_t(X_t)$.
\end{enumerate}
\end{mydefn}
The set of all feedback controls $\Theta$ corresponding to the initial value $(s,\mu)$ is denoted by $\Pi_{s,\mu}$.

\begin{myrem}\label{rem-1}
1).\ In Definition \ref{def-1}, we simply demand $(\alpha_t)_{t\in[s,T]}$ is measurable and $\F_t$-adapted, which implies that $(\alpha_t)_{t\in [s,T]}$ has a progressively measurable modification due to \cite[Proposition 1.12, p.5]{KaS}. Moreover,  according to  \cite[Remark 1.1, p.45]{Ikeda},    $(\alpha_t)_{t\in [s,T]}$ even has a modification to be predictable by noting that
$\E\Big[\int_0^T\!\Big(\int_U|z|\alpha_t(\d z)\Big)^2\d t\Big]<\infty$,
since $U$ is a compact set.

\noindent 2). According to measure theory, if $\alpha_t$ is measurable with respect to the $\sigma$-algebra generated by $X_t$, denoted by $\sigma\{X_t\}$, then there exists a measurable map $G_t:\R^d\to\pb(U)$ such that $\alpha_t= G_t(X_{t})$. Hence, to verify item $\mathrm{(iv)}$ in Definition \ref{def-1}, we only need to check that $\alpha_t$ is $\sigma\{X_t\}$ measurable.
%
\end{myrem}




%

Given two measurable functions $f:[0,T]\times \R^d\times \pb(\R^d)\times \pb(U)\to \R$ and $g:\R^d\times \pb(\R^d)\to \R$, the objective function is defined by
\begin{equation}\label{a-3}
J(s,\mu;\Theta)=\E_{s,\mu}\Big[\int_s^{T } f(r,X_r,\law_{X_r},\alpha_r)\d r+g(X_{T},\law_{X_{T}})\Big]
\end{equation} for $s\!\in \![0,T)$ and $\mu\!\in\! \pb(\R^d)$, where $\E_{s,\mu}$ means taking expectation w.r.t.\,the initial value $\law_{X_s}\!=\!\mu$.
The value function is defined by
\begin{equation}\label{a-4}
V(s,\mu)=\inf_{\Theta\in \Pi_{s,\mu}} J(s,\mu;\Theta).
\end{equation}
A feedback control $\Theta^\ast\in \Pi_{s,\mu}$ is said to be optimal, if it satisfies
 $J(s,\mu;\Theta^\ast)=V(s,\mu).$


Notice that under the uniqueness condition of Definition \ref{def-1}(iii), $J(s,\mu;\Theta)$ and hence $V(s,\mu)$ are well defined.
Namely, for any two given  random variables $\xi,\tilde \xi$ with $\law_\xi\!=\!\law_{\tilde \xi}\!\in \!\pb(\R^d)$,  the values of $J(s,\mu;\Theta)$ and $V(s,\mu)$ will be fixed no matter which initial value $X_s=\xi$ or $X_s=\tilde \xi$ is used to determine the solution of SDE \eqref{a-1} for the process $(X_t)$ in \eqref{a-3}.

Indeed,
noticing \[\E[g(X_T,\law_{X_T})]=\int_{\R^d} g(x,\law_{X_T}) \law_{X_T}(\d x),\] we see that the term $\E[g(X_T,\law_{X_T})]$ in \eqref{a-3} depends only on the distribution of $X_T$, which is determined by the law $\mu$ of $\xi$ due to Definition \ref{def-1}(iii).
Since $\Theta\in \Pi_{s,\mu}$ is a feedback control, for each $t\in [s,T]$, there exists a functional $F_t$ such that $\alpha_t=F_t(X_{ t })$. The distribution of $X_{  t }$  is uniquely determined by the distribution $\mu$ of $\xi$ thanks to the uniqueness in law for SDE \eqref{a-1}.   Then, from the representation
\begin{align*}
  &\E\Big[\int_s^T \!f(r,X_r,\law_{X_r}, \alpha_r)\d r\Big]=\int_s^T\! \int_{ \R^d }\!f(r, x_{ r } ,\law_{X_r},  F_r(x_{r}))\law_{X_r}(\d x_{ r })\d r,
\end{align*}
we have that the running cost in \eqref{a-3} is also uniquely determined by $\mu$.

%

We introduce the conditions on the controlled system used below.
\begin{itemize}
  \item[$\mathrm{(H1)}$] There exists a constant $K_1$ such that
      \[|b(t,x,\nu_1, \alpha)-b(s,y,\nu_2, \alpha)|
      +\|\sigma(t,x,\nu_1)-\sigma(s,y,\nu_2)\|\leq K_1\big(|t-s|+|x-y|+\W_1(\nu_1,\nu_2)\big)\]
      for all $s,t\in [0,T]$, $x,y\in \R^d$, $\nu_1,\nu_2\in \pb_1(\R^d)$,   and $\alpha\in \pb(U)$.
  \item[$\mathrm{(H2)}$] There exists a constant $K_2$ such that
      \[|b(t,x,\mu, \alpha)|\!+\!\|\sigma(t,x,\mu)\|
      \leq K_2\Big(1\!+\!|x|\!+\!\int_{\R^d}\! |z|\,\mu(\d z)\Big)
       \]  for $t\in [0,T], x\in \R^d,   \mu\in \pb_1(\R^d), \alpha\in \pb(U).$
  \item[$\mathrm{(H3)}$] There exist constants $K_3$ and $K_4$ such that
      \begin{gather*}
      |f(t,x,\mu,\alpha)-f(s,y,\nu,\alpha)|+ |g(x,\mu)-g(y,\nu)| \leq K_3\big(|t-s|+|x-y|+\W_1(\mu,\nu)\big),\\
      |f(t,x,\mu,\alpha)|+|g(x,\mu)|\leq K_4\Big(1+|x|+\int_{\R^d}|z|\mu(\d z)\Big)
      \end{gather*}
      for $t,s\in[0,T]$, $x,y\in \R^d$, $\mu,\nu\in \pb_1(\R^d)$, and $\alpha\in \pb(U)$.
\end{itemize}



The conditions (H1) and (H2) are used to ensure the existence of solution to SDE \eqref{a-1} under suitable control process $(\alpha_t)$. For example, under (H1) and (H2), for each $\alpha\in \pb(U)$, we consider the control strategy $\alpha_t\equiv \alpha\in \pb(U)$, then SDE \eqref{a-1} admits a unique weak solution for any initial value $X_s=\xi$ with $\law_\xi\in \pb_1(\R^d)$; see, e.g.
\cite[Theorems 2.1 and 3.1]{Fun}. Moreover, if assume, in addition, $\sigma(t,x,\mu)$ depends only on $t,x$, then SDE \eqref{a-1} admits a unique strong solution for $\alpha_t\equiv \alpha$ due to \cite[Theorem 2.1]{Wang18}.  So, there is an admissible control $\Theta$ associated with such deterministic control strategy $\alpha_t\equiv \alpha$.
Consequently, $\Pi_{s,\mu}$ is not empty.

\begin{myrem}\label{rem-2.1}
We notice that under conditions $\mathrm{(H1)}$, $\mathrm{(H2)}$, it does not mean that for any $\F_t$-adapted process $(\alpha_t)$ integrable in $t$ over  $[s,T]$ with $T>1$, the corresponding SDE \eqref{a-1} will admit a solution with initial value $X_s=\xi$ and $\law_\xi\in \pb_1(\R^d)$. The process $(t,\omega)\mapsto \alpha_t(\omega)$ has essential impact on the wellposedness of the associated SDE \eqref{a-1}.

Consider SDE \eqref{a-1} with the coefficients $\sigma(t,x,\nu)=0$, $b(t,x,\nu, \alpha)=\int_{[-T,T]} x \alpha(\d x)$  for $(t,x,\nu,\alpha)\in [0,T]\times\R^d\times\pb(\R^d)\times \pb(U)$ with $U=[-T,T]$. Clearly, $\mathrm{(H1)}$, $\mathrm{(H2)}$ hold. Let $\alpha_t=-\delta_{\mathrm{sgn} X_t}$, where $\delta_x$ denotes the Dirac measure over $x$ and $\mathrm{sgn}(x)=1$ if $x>0$; $=-1$, if $x\leq 0$. Then, $\alpha_t$ is a bounded functional of $X_t$, which yields that $\alpha_t$ is $\F_t$-adapted and $\E\int_0^T|\alpha_r(|\cdot|)|^2\d r<\infty$. In this case,
the SDE \eqref{a-1} turns into
\begin{equation}\label{a-5}
\d X_t=-\mathrm{sgn}(X_t)\d t.
\end{equation}
However, if we consider the initial value $X_0=0$, then according to \cite[Example 1.16]{Cherny},   equation \eqref{a-5} has no (weak or strong) solution. If $X_0=a\neq 0$, equation \eqref{a-5} admits a solution up to the hitting time of $0$. Moreover, this example also tells us the set $\Pi_{s,\mu}$ of feedback controls under the conditions $\mathrm{(H1)}$ and $\mathrm{(H2)}$ also depends on $s$ and $\mu$. 
\end{myrem}

\section{Existence of Optimal Feedback Control}
 This section is devoted to showing the existence of optimal feedback controls. We shall use the compactification method which was proposed by Kushner \cite{Ku72}, and has been developed in many works such as \cite{DM,FS21,Haus1,Haus}; see the survey  \cite{Ku14} for more references.  In this part,  we shall generalize this method to deal with the optimal control problem for the  McKean-Vlasov equations.

Let $\mathscr{U}$ be the collection of maps $\mu: [0,T]\to \pb(U)$ such that for any $A\in \mathscr{B}(U)$, $t\mapsto \mu_t(A)$ is integrable relative to $\d t$ on $[0,T]$. $\mathscr{U}$ can be viewed as a subspace of $\pb([0,T]\!\times \!U)$ through the map
 \[(\mu_t)_{t\in [0,T]}\mapsto \bar{\mu}, \quad \]
 \text{where $\bar\mu$ is defined by}: for any $A_1\in \mathscr{B}([0,T])$, $A_2\in \mathscr{B}(U)$,
 \[\bar\mu(A_1\times A_2)=\frac1T\int_{A_1}\mu_t(A_2)\d t.\]
Endow $\mathscr{U}$ with the weak topology induced from $\pb([0,T]\!\times\! U)$.
Since $[0,T]\!\times\! U$ is compact, $\pb([0,T]\!\times \!U)$ is compact too. $\mathscr{U}$ is also compact as a closed set in $\pb([0,T]\!\times\! U)$.

As done in \cite{Haus1}, it is convenient to transform the feedback controls into the canonical space. Let
\begin{equation}\label{c-1}
\mathcal{Y}=\C([0,T];\R^d)\times\C([0,T];\R^d)\times \mathscr{U},
\end{equation} which is endowed with the product topology.  Let $\wt{\mathcal{Y}}$ be the Borel $\sigma$-field, $\wt{\mathcal{Y}}_t$ the $\sigma$-fields up to time $t$. For each feedback control $\Theta=(\Omega,\F,\p,\{\F_t\}_{t\geq 0},W_\cdot$, $ X_\cdot,\alpha_\cdot) \in \Pi_{s,\mu}$, define a map $\Psi_\alpha:\Omega\to \mathcal{Y}$ by
\begin{equation}\label{c-2}
\Psi_\alpha(\omega)=(W_t(\omega),X_t(\omega),\alpha_t( \omega))_{t\in [0,T]},
\end{equation}
where $X_t(\omega):=X_s(\omega)$, $\alpha_t(\omega):= \alpha_s(\omega)$ for $t\in [0,s]$.
Let $R=\p\circ \Psi_\alpha^{-1}$ be the induced probability measure over the canonical space $\mathcal{Y}$, then
\begin{equation}\label{c-3}
J(s,\mu;\Theta)=J(s,\mu;R):=\E_R\Big[\int_s^T f(r, X_r,\law_{X_r},\alpha_r)\d r+g(X_T,\law_{X_T})\Big|\law_{X_s}=\mu\Big].
\end{equation}
Put
\[\mathcal{R}_{s,\mu}=\big\{R=\p\circ \Psi_\alpha^{-1}; \ \Theta\in\Pi_{s,\mu}\big\}.\]
Consequently, under the transform $\Psi_\alpha$, we have
\[V(s,\mu)=\inf_{\Theta\in \Pi_{s,\mu}} J(s,\mu;\Theta)= \inf_{R\in \mathcal{R}_{s,\mu}} J(s,\mu;R).\] Through the transform $\Psi_\alpha$, every $\Theta\in \Pi_{s,\mu}$ corresponds to a probability measure $R\in \mathcal{R}_{s,\mu}$, and vice versa. Therefore, we only need to show the existence of an $R^\ast\in \mathcal{R}_{s,\mu}$ such that $V(s,\mu)=J(s,\mu;R^\ast)$ to ensure the existence of optimal control $\Theta^\ast\in \Pi_{s,\mu}$. See \cite[Section 2.2]{Haus1} for more details.

As explained in Subsection 2.1, the uniqueness in law for McKean-Vlasov SDE \eqref{a-1} plays a crucial role in the wellposedness of the value function $V(s,\mu)$. However, it is not an easy task to ensure the uniqueness in law for SDE \eqref{a-1} due to the dependence of distribution of  the diffusion coefficient. For instance, there are examples that the uniqueness in law may fail even if the drift $b$ is bounded and    $\sigma\sigma^\ast$   is uniformly elliptic; see, \cite[Example 3]{Zhao}. This is quite different to the classical SDEs whose coefficients do not rely on the distribution  (cf. Stroock and Varadhan \cite{SV79}). Next, we shall use the result in \cite{Ray} to ensure the uniqueness in law of SDE \eqref{a-1}. To this aim, let us first introduce the notion of linear functional derivative.

\begin{mydefn}\label{def-2}
The continuous function $h:\pb(\R^d)\to \R$ is said to own a \emph{linear functional derivative} if there exists a real valued bounded measurable function
\[\pb(\R^d)\times\R^d\ni (\mu,x)\mapsto \partial_\mu h (x)\in \R\]
such that for all $x\in \R^d$, the map $\pb(\R^d)\ni\mu\mapsto \partial_\mu h (x)$ is continuous and for all $\mu,\nu\in \pb(\R^d)$, it holds
\begin{equation}\label{c-3.5}
\lim_{\veps\downarrow 0} \frac{h\big((1-\veps)\mu+\veps\nu\big)-h(\mu)}{\veps} =\int_{\R^d} \partial_\mu h (y)\,\d (\mu-\nu)(y).
\end{equation} The map $y\mapsto \partial_\mu h(\mu)(y)$ is unique determined up to an additive constant.
\end{mydefn}
Such a notation of derivatives was introduced in \cite{Car1}, see also \cite{CDLL}. It was used in \cite{Ray} to study the martingale problem for McKean-Vlasov SDEs.

Let us introduce a modification of the Wasserstein distance for $\mu,\nu\in \pb(\R^d)$, that is,
\begin{equation}\label{W-2}
\overline{\W}_1(\mu,\nu):=\inf_{\Gamma\in \mathcal{C}(\mu,\nu)} \int_{\R^d\times\R^d}\!\! (1\wedge|x-y| )\Gamma(\d x,\d y).
\end{equation}
It is clear that $\overline{\W}_1(\mu,\nu)\leq \W_1(\mu,\nu)$ for any $\mu,\nu\in \pb(\R^d)$. Also, the convergence in $\overline{\W}_1$ is equivalent to the weak convergence in $\pb(\R^d)$.  Moreover, according to \cite[Theorem 6.15]{Vil}, by the boundedness of $1\wedge |x-y|$,
\begin{equation}\label{W-3}
\overline{\W}_1(\mu,\nu)\leq  \|\mu-\nu\|_{\var},\quad \mu,\nu\in \pb(\R^d).
\end{equation}

After these preparations, we can introduce  the following condition in order to guarantee the uniqueness in law for the solution to McKean-Vlasov SDE \eqref{a-1} under feedback controls.
\begin{itemize}
  \item[$\mathrm{(H4)}$]
  \begin{itemize}
  \item[(i)] The drift $b(t,x,\mu,\alpha)$ is   bounded, $\R^d\!\times\!\pb(U)\!\ni\!(x,\alpha)\mapsto b(t,x,\mu,\alpha)$ is continuous, and there exists a constant $C>0$ such that for all $t\in [0,T]$, $x\in \R^d$, $\alpha\in  \pb(U) $,
      \[|b(t,x,\mu,\alpha)-b(t,x,\nu,\alpha)|\leq C \overline{\W}_1(\mu,\nu),\quad \mu,\nu\in \pb(\R^d),\]
      \item[(ii)] The coefficient $[0,T]\!\times\!\R^d\!\times\!\pb(\R^d)\!\ni\! (t,x,\mu)\mapsto \sigma(t,x,\mu)$ is a bounded continuous function with $\pb(\R^d)$ endowed with the distance $\overline{\W}_1$.  For any $(t,\mu)\in [0,T]\times\pb(\R^d)$, $\R^d\ni x\mapsto a(t,x,\mu):=(\sigma\sigma^\ast)(t,x,\mu)\in \R^d\times\R^d$ is uniformly $\eta$-H\"older continuous for some $\eta\in (0,1]$, i.e.
      \[\sup_{t\in[0,T],\mu\in\pb(\R^d), x\neq y}\frac{ |a(t,x,\mu)-a(t,y,\mu)|} {|x-y|^{\eta}}<\infty.\]
      \item[(iii)] For any $i,j\in \{1,2,\ldots,d\}$  and any $(t,x)\in [0,T]\times\R^d$, $\pb(\R^d)\ni\mu\mapsto a_{ij}(t,x,\mu)$ has a linear functional derivative.
      \item[(iv)] For any $i,j\in \{1,2,\ldots,d\}$ and any $(t,\mu)\in [0,T]\times \pb(\R^d)$, $\R^d\times\R^d\ni (x,y)\mapsto \partial_\mu a_{ij}(t,x,\mu)(y)$ is an $\eta$-H\"older continuous function for some $\eta\in (0,1]$, uniformly with respect to $t$ and $\mu$.
      \item[(v)] There exists a constant $\lambda >1$ such that for any $(t,\mu)\in [0,T]\times \pb(\R^d)$, for any $(x,z)\in \R^d\times \R^d$,
          \[\lambda^{-1}|z|^2\leq \la a(t,x,\mu)z,z\raa\leq \lambda |z|^2.\]
  \end{itemize}
\end{itemize}

\begin{mylem}[\cite{Ray}, Theorem 3.4]\label{lem-4}
Assume  $\mathrm{(H4)}$ holds, and $\alpha_t=F_t(X_{t})$ for some measurable functional $F_t: \R^d \to \pb(U)$ for $t\in [s,T]$. Then the uniqueness in law holds for SDE \eqref{a-1}.
\end{mylem}

Compared with condition (H1), the condition (H4)(i) does not needs to consider the regularity of $x\mapsto b(t,x,y,\alpha)$, and hence is more suitable to be used to study the  uniqueness in law for the controlled SDE \eqref{a-1}. Indeed,  we can  view the controlled SDE \eqref{a-1} as a distribution dependent SDE with drift $\tilde b(t,x,\mu):=b(t,x,\mu, F_t(x))$ associated with a feedback control $\alpha_t=F_t(X_t)$.
Therefore, in this setting for the uniqueness in law of SDE \eqref{a-1}, one only needs to check the Lipschitz continuity of $\mu\mapsto b(t,x,\mu,\alpha)$. This simplifies greatly the verification of item (iii) in Definition \ref{def-1} of a Markovian feedback control. Otherwise, it is a quite challenge task to study the uniqueness in law for   distribution dependent SDEs associated with general feedback controls.

\begin{mythm}[Existence of optimal control]\label{thm-1}
Suppose conditions $\mathrm{(H3)}$, $\mathrm{(H4)}$  hold. Then, for each $(s,\mu)\in [0,T]\times \pb_p(\R^d)$ for some $p>1$, there exists an optimal Markovian feedback control $\Theta^\ast\in \Pi_{s,\mu}$.
\end{mythm}

\begin{proof}
  We only need to consider the nontrivial case that $V(s,\mu)<\infty$. Otherwise, every Markovian feedback control $\Theta\in \Pi_{s,\mu}$ will be optimal. For simplicity of notation, we only consider the case $s=0$. The proof is separated into three steps.

  \noindent\textbf{Step 1}. We aim to show the tightness of a minimizing sequence. According to the definition of $V$, there exists a sequence of feedback controls $\Theta_n\in \Pi_{s,\mu}$ with associated control strategies $\alpha^n_\cdot$ such that
  \begin{equation}\label{c-4}
  \lim_{n\to \infty} J(0,\mu;\Theta_n)=V(0,\mu)<\infty.
  \end{equation}
  Setting $R_n=\p\circ\Psi_{\alpha^n}^{-1}$, we shall prove the tightness of $(R_n)_{n\geq 1}$. Recall that
  \[\Psi_{\alpha^n}(\omega)=(W^n_\cdot(\omega), X^n_\cdot(\omega), \alpha_\cdot^n(\omega))\in \mathcal{Y}.\]
  Let $\law_{W^n},\,\law_{X^n},\,\law_{\alpha^n}$ be the marginal distributions of $R_n$, $n\geq 1$. Since the distributions of all $\law_{W^n}$ are the same, that is, the distribution of $d$-dimensional Brownian motion, therefore, the tightness of $(\law_{W^n})_{n\geq 1}$ is obvious. Moreover, since $ \law_{\alpha^n},\, n\geq 1 $ are distributions in the compact space $\mathscr{U}$, $(\law_{\alpha^n})_{n\geq 1}$ is tight too. Now what we need to show is the tightness of $(\law_{X^n})_{n\ge 1}$.

  As $(W_t, X_t,\alpha_t)$ satisfies SDE \eqref{a-1}, it holds
  \[\E|X^n_t-X^n_s|\leq \E\Big[\int_s^t|\la b(r,X^n_r,\law_{X^n_r},\cdot), \alpha^n_r\raa |\d r\Big]+\E\Big|\int_s^t\sigma(r,X^n_r,\law_{X^n_r})\d W^n_r\Big|.\]
  Due to the boundedness of the coefficients $b$ and $\sigma$,   there exists constant $C>0$ such that
  \[\E|X_t^n-X^n_s|\leq C(|t-s|+ \sqrt{|t-s|}),\quad t,s\in [0,T], \ n\geq 1.\]
  According to \cite[Theorem 12.3]{Bill}, the distributions of $(X_\cdot^n)_{n\geq 1}$ are tight, i.e. $(\law_{X^n})_{n\geq 1}$ is tight.
  Because all the marginal distributions of $(R_n)_{n\geq 1}$ are tight, so is $(R_n)_{n\geq 1}$. Hence, there exists a subsequence $(R_{n_k})_{k\ge 1}$ of $(R_n)_{n\geq 1}$ and a probability measure $R_0$ on $\mathcal{Y}$ such that
  \[ R_{n_k}\ \text{weakly converges to $R_0$ as $k\to \infty$}.\]

\noindent\textbf{Step 2}. In this step, we aim to show the limit $R_0$ of $R_{n_k}$ is also a Markovian feedback control, which will be the desired optimal Markovian feedback control.

According to Skorokhod's representation theorem (cf. e.g. \cite[Theorem 1.8, p.102]{EK}), there exists a probability space $(\Omega',\F',\p')$ on which defined a sequence of random variables $Y_{n_k}=(W_t^{n_k}, X_t^{n_k}, \alpha_t^{n_k})_{t\in [0,T]}\in \mathcal{Y}$, $k\geq 1$, and $Y_0=(W_t^{(0)},X_t^{(0)}, \alpha_t^{(0)})_{t\in [0,T]}\in \mathcal{Y}$, with the distribution $R_{n_k}$ and $R_0$ respectively such that
\begin{equation}\label{c-5}
\lim_{k\to \infty} Y_{n_k}=Y_0,\quad \quad \p'\text{-a.s.}.
\end{equation}

Next, we go to show that $Y_0=(W_t^{(0)}, X_t^{(0)},\alpha_t^{(0)})_{t\in [0,T]}$ is associated with a feedback control in $\Pi_{0,\mu}$.
To this purpose, we need to prove three assertions below according to Definition \ref{def-1}.
\begin{itemize}
  \item[(i)] $(W_t^{(0)}, X_t^{(0)},\alpha_t^{(0)})_{t\in [0,T]}$ satisfies the SDE
      \begin{equation} \label{c-6}
      X_t^{(0)}\!=X_0^{(0)}\!+\!\int_0^t\!\la b(r,X_r^{(0)}, \law_{X_r^{(0)}},\cdot),\alpha_r^{(0)}\raa \d r\!+\!\int_0^t\!\! \sigma(r,X_r^{(0)},\law_{ X_r^{(0)}})\d W_r^{(0)},\quad t\in\! [0,T].
      \end{equation}
  \item[(ii)] $\alpha_t^{(0)}$ is $\sigma\{X_t^{(0)}\}$ measurable, $0\leq t\leq T$.
  \item[(iii)] The uniqueness in law holds for SDE \eqref{c-6}.
\end{itemize}
Under (H4), Lemma \ref{lem-4} implies that (iii) the uniqueness in law holds for SDE \eqref{c-6}. Up to taking a subsequence of $n_k$, by the bounded continuous property of $b(t,x,\mu,\alpha)$, $\sigma(t,x,\mu)$, and the almost sure convergence of $(W_\cdot^{n_k},X_\cdot^{n_k},\alpha_\cdot^{n_k})\in \mathcal{Y}$ to $(W_\cdot^{(0)},X_\cdot^{(0)},\alpha_\cdot^{(0)})$, assertion (i) follows from the fact
\begin{equation}\label{c-7}
X_t^{n_k}=X_0^{n_k}+\int_0^t\!  b(r,X_r^{n_k},\law_{X_r^{n_k}},\alpha_r^{n_k}) \d r +\int_0^t\! \sigma(r, X_r^{n_k}, \law_{X_r^{n_k}})\d W_r^{n_k},\quad t\in[0,T].
\end{equation}
This is the crucial part of the proof, and its argument is  delicate and cumbersome, so it is deferred to the Lemma \ref{lem-5} below. \eqref{c-6} follows from \eqref{c-7} according to Lemma \ref{lem-5} by taking a subsequence of $(n_k)$ if necessary.

Now we show assertion (ii). We adopt the notation in the study of backward martingale  to define
\[\F_{-k,t}^X=\sigma\{ X_t^{n_m}; m \geq k\}, \quad \F_{-k,t}^\alpha=\sigma\{\alpha_t^{n_m}; m\geq k\}, \quad t\in [0,T]. \]
Then
\[\F_{-1,t}^X \supset \F_{-2,t}^X\supset\cdots\supset \F_{-k,t}^X\supset \F_{-k-1,t}^X\supset\cdots. \]
Put $ \F_{-\infty,t}^X=\bigcap_{k\geq 1}\F_{-k,t}^X$. Then $\F_{-\infty,t}^X$ is a $\sigma$-field, and it concerns only the limit behavior of the sequence $(X_t^{n_k})_{k\ge 1}$. Since $\lim_{k\to \infty} X_t^{n_k}=X_t^{(0)}$ a.s. for every $t\in [0,T]$, we have
\begin{equation}\label{c-8}
\F_{-\infty, t}^X=\sigma\{X_t^{(0)}\}.
\end{equation}
Since $\alpha_t^{n_k}\in \sigma\{X_t^{n_k}\}$, $k\geq 1$,  it holds $\F^\alpha_{-k,t} \subset \F_{-k,t}^{X}$. Using Lemma \ref{lem-5} below, the almost sure convergence of  $\alpha_t^{n_k}$ to $\alpha_t^{(0)}$ as $k\to \infty$ yields that
\[\sigma\{\alpha_t^{(0)}\}\subset \bigcap_{k\geq 1}\F_{-k,t}^\alpha\subset \bigcap_{k\geq 1}\F_{-k,t}^X=\F_{-\infty,t}^X=\sigma\{X_t^{(0)}\}.\]
Therefore, $\alpha_t^{(0)}$ is $\sigma\{X_t^{(0)}\}$ measurable,  and hence there exists a measurable map $F_t:\R^d\to \pb(U)$ such that $\alpha_t^{(0)}=F_t(X_t^{(0)})$. Let $\F_t'=\sigma\big\{(W_r^{(0)}, X_r^{(0)},\alpha_r^{(0)}) ;\ 0\leq r\leq t\big\}$, $t\in[0,T]$. As a consequence, $\Theta^{(0)}=(\Omega',\F',\p', \{\F_t'\}_{t\geq 0}, W_\cdot^{(0)},X_\cdot^{(0)},\alpha_\cdot^{(0)}) $ is a Markovian feedback control.

\noindent \textbf{Step 3}.
By the boundedness of $b$ and $\sigma$, and $\mu$ in   $\pb_p(\R^d)$ for some $p>1$, there exists a constant $c_p$ such that
\begin{align*}
  \E|X_t^{(n_k)}|^p&\leq c_p\Big(\E|X_s^{(n_k)}|^p+\E\Big|\int_s^t b(r,X_r^{(n_k)},\law_{X_r^{(n_k)}}, \alpha_r^{(n_k)})\d r\Big|^p\\
  &\qquad +\E\Big(\int_s^t \|\sigma(r,X_r^{(n_k)},\law_{X_r^{(n_k)}})\|^2\d r\Big)^{\frac p2}\Big)\\
  &\leq c_p\big(\mu(|\cdot\,|^p)+ |b|^p_\infty(T-s)^p+\|\sigma\|^p_\infty (T-s)^{\frac{p}2}\big).
\end{align*}
Hence, $(X_t^{(n_k)})_{k\geq 1}$ is uniformly integrable in $L^1(\p)$. Together with \eqref{c-5}, this yields that
\begin{equation}\label{c-8.5}\lim_{k\to \infty}\W_1(\law_{X_t^{(n_k)}},\law_{X_t^{(0)}})\leq \lim_{k\to \infty}\E|X_t^{(n_k)}-X_t^{(0)}|=0.
\end{equation}
Then, by (H3) and \eqref{c-4},
\begin{align*}
V(0,\mu)&=\lim_{k\to \infty} J(0,\mu;\Theta_{n_k})\\
 &=\lim_{k\to \infty} \E_{\p'}\Big[\int_0^T \!\!f(r, X_r^{n_k}, \law_{X_r^{n_k}},\alpha_{r}^{n_k})\d r+ g(X_T^{n_k},\law_{X_T^{n_k}})\Big]\\
 &= \E_{\p'}\Big[\int_0^T\!\!f(r,X_r^{(0)},\law_{X_r^{(0)}}, \alpha_r^{(0)})\d r+g(X_T^{(0)}, \law_{X_T^{(0)}})\Big]\\
 &\geq V(0,\mu).
\end{align*}
Therefore, $\Theta^{(0)}$ is an optimal Markovian feedback control. The proof is complete.
\end{proof}

\begin{mylem}\label{lem-5}
Adopting the conditions and notations of Theorem \ref{thm-1}, we have:
\begin{itemize}
  \item[$\mathrm{(i)}$] $\dis  \lim_{k\to \infty} \int_0^t\la b(r,X_r^{n_k}, \law_{X_r^{n_k}},\cdot),\alpha_r^{n_k}\raa \d r
  =\int_0^t\la b(r,X_r^{(0)},\law_{X_r^{(0)}}, \cdot),\alpha_r^{(0)}\raa \d r,$ a.s..
  \item[$\mathrm{(ii)}$]
  $\dis \lim_{k\to \infty}\E\Big[\Big|\int_0^t\sigma(r,X_r^{n_k},\law_{ X_r^{n_k}})\d W_r^{n_k}-\int_0^t\sigma(r,X_r^{(0)}, \law_{X_r^{(0)}})\d W_r^{(0)}\Big|^2\Big]=0.$
\end{itemize}
\end{mylem}

\begin{proof}
  (i) As $\alpha_\cdot^{n_k}$ converges almost surely to $\alpha^{(0)}_\cdot$ in $\mathscr{U}$, this yields that for any bounded continuous function $[0,T]\times U\ni(r,z)\mapsto h(r,z),$
  \begin{equation}\label{c-9}
  \lim_{k\to \infty} \int_0^T\int_U\!h(r,z) \alpha_r^{n_k}(\d z)\d r =\int_0^T\!\int_U\! h(r,z) \alpha_r^{(0)}(\d z)\d r,
  \quad \p'\text{-a.s.}.
  \end{equation}
We shall show that, for any   $\phi\in C_b(U)$, for every $0\leq s<t\leq T$,
\begin{equation} \label{c-17}
\begin{split}
\lim_{k\to \infty}\! \int_s^t\! \!\int_U\phi(z)\alpha_r^{n_k}(\d z)\d r&=\lim_{k\to \infty}\!\int_0^T\!\!\int_U\!\mathbf1_{[s,t]}(r) \phi(z)\alpha_r^{n_k}(\d z)\d r\!\\
&=\! \int_s^t\!\!\int_U\phi(z)\alpha_r^{(0)} (\d z)\d r,\  \quad  \p'\text{-a.s.}.
\end{split}
\end{equation}
  Since the indicator function $r\mapsto\mathbf1_{[s,t]}(r)$ is not continuous, we use smooth functions to approximate it.
  For any $m\in \N$, there exists a continuous function $\beta_m:[0,T]\to [0,1]$ with support in $[s,t]$ such that
  \[{Leb}\big\{r\in [0,T]; |\beta_m(r) -\mathbf1_{[s,t]}(r)|>0\big\} <\frac{1}{2m|\phi|_\infty},
  \] where $|\phi|_\infty=\sup_{z\in U}|\phi(z)|<\infty$ and $ {Leb}$ denotes the Lebesgue measure over $\R$.
  We derive that
  \begin{equation}\label{c-10}
    \begin{split}
     &\Big|\int_s^t\!\int_U\!\phi(z)\alpha_r^{(0)}(\d z) \d r-\int_0^T\!\int_U\beta_m(r)\phi(z)\alpha_r^{(0)} (\d z)\d r\Big|\leq \frac 1m,\\
     &\Big|\int_s^t\!\int_U\!\phi(z)\alpha_r^{n_k}(\d z)\d r-\int_0^T\!\int_U\!\beta_m(r)\phi(z)\alpha_r^{n_k} (\d z)\d r\Big|\leq \frac 1m, \ k\geq 1.
    \end{split}
  \end{equation}
  By \eqref{c-9}, there exists $K\in \N$ such that for any $k\geq K$
  \begin{equation}\label{c-11}
  \Big|\int_0^T\!\int_U\beta_m(r)\phi(z)\alpha_r^{n_k} (\d z)\d r-\int_0^T\!\int_U\!\beta_m(r)\phi(z)\alpha_r^{(0)} (\d z)\d r\Big|\leq \frac 1m,\quad \p'\text{-a.s.}.
  \end{equation}
  Combining \eqref{c-11} with \eqref{c-10}, we get that for any $k\geq K$,
  \begin{equation*}
    \Big|\int_s^t\!\int_U\! \phi(z)\alpha_r^{n_k} (\d z)\d r-\int_s^t\!\int_U\!\phi(z)\alpha_r^{(0)} (\d z)\d r\Big|\leq \frac{3}{m},\quad \p'\text{-a.s.}.
  \end{equation*}
  The equation \eqref{c-17} follows immediately by letting $m\to \infty$.

  Consequently, assertion (i) follows from \eqref{c-17}, $X_t^{n_k}\to X_t^{(0)}$ a.s. as $k\to\infty$, and the bounded continuity of $(x,\mu,\alpha)\mapsto b(t,x,\mu,\alpha)$.

  (ii) 
 Note that
  \begin{align*}
    &\E\Big|\int_0^t\!\sigma(r,X_r^{n_k},\law_{X_r^{n_k}}) \d W_r^{n_k}-\int_0^t\!\sigma(r,X_r^{(0)},\law_{X_r^{ (0)}})\d W_r^{(0)}\Big|^2\\
    &\leq 2\E\Big|\int_0^t\!\sigma(r,X_r^{n_k}, \law_{X_r^{n_k}}) \d W_r^{n_k}-\int_0^t\!\sigma(r,X_r^{(0)},\law_{X_r^{ (0)}})\d W_r^{n_k}\Big|^2 \\
    &\quad +2\E\Big|\int_0^t \!\sigma(r,X_r^{(0)},\law_{X_r^{ (0)}})\d W_r^{n_k}- \int_0^t\!\sigma(r,X_r^{(0)},\law_{X_r^{ (0)}})\d W_r^{(0)}\Big|^2\\
    &=:2(I+I\!I).
  \end{align*}
  It follows from  the dominated convergence theorem, \eqref{c-8.5}, and the continuity of $\sigma$ that
  \begin{equation}\label{c-12}
  \lim_{k\to \infty} I =\lim_{k\to \infty}\E\Big[\int_0^t\!\big|\sigma(r, X_r^{n_k},\law _{X_r^{n_k}})-\sigma(r,X_r^{(0)},\law_{X_r^{(0)}}) \big|^2\d r\Big]=0.
  \end{equation}
  To deal with the term $I\!I$, we use the time discretization method. For an integer $N\geq 1$, let $t_l=\frac{l}{N} t$, $l=0,1,\ldots,N$, and
  \[\eta_N(t)=t_l, \quad \text{if $t\in [t_l, t_{l+1})$}.\]
  Then,
  \begin{equation}\label{c-13}
  \begin{split}
    I\!I&\leq 3\Big\{\E\Big|\sum_{l=0}^{N-1}\int_{t_l}^{t_{l+1}}\! \big(\sigma(r, X_r^{(0)},\law_{X_r^{(0)}})-\sigma (t_l,X_{t_l}^{(0)}, \law_{X_{t_l}^{(0)}}) \big)\d W_r^{n_k}\Big|^2\\
    &\quad +\E\Big|\sum_{l=0}^{N-1}\int_{t_l}^{t_{l+1}}\!\! \big(\sigma(r, X_r^{(0)},\law_{X_r^{(0)}})-\sigma (t_l,X_{t_l}^{(0)}, \law_{X_{t_l}^{(0)}}) \big)\d W_r^{(0)}\Big|^2\\
    &\quad+\E\Big|\sum_{l=0}^{N-1}\int_{t_l}^{t_{l+1}}\!\! \sigma(t_l,X_{t_l}^{(0)}, \law_{X_{t_l}^{(0)}}) \d\big( W_r^{n_k}-W_r^{(0)}\big)\Big|^2\Big\}
    \\
    &= 3\Big\{2\E\Big[\int_0^T\!\!\big|\sigma(r, X_r^{(0)},\law_{X_r^{(0)}})-\sigma(\eta_N(r), X_{\eta_N(r)}^{(0)},\law_{X_{\eta_N(r)}^{(0)}}) \big|^2\d r\Big]\\
    &\quad+\E\Big|\sum_{l=0}^{N-1}\int_{t_l}^{t_{l+1}} \!\! \sigma(t_l, X_{t_l}^{(0)}, \law_{X_{t_l}^{(0)}})\d\big(W_r^{n_k}-W_r^{(0)}\big)
    \Big|^2\Big\}.
  \end{split}
  \end{equation}
  Using the dominated convergence theorem again, by virtue of  the continuity of the function $\sigma(t,x,\mu)$ and the paths of $(X_t^{(0)})_{t\in [0,T]}$,
  \begin{equation}\label{c-14}
  \lim_{N\to \infty} \E\Big[\int_0^T\!\!\big|\sigma(r, X_r^{(0)},\law_{X_r^{(0)}})-\sigma(\eta_N(r), X_{\eta_N(r)}^{(0)},\law_{X_{\eta_N(r)}^{(0)}}) \big|^2\d r\Big]=0.
  \end{equation}
  Applying the independent increment property of Brownian motion and the fact $X_t^{(0)}$ is $\F_t$-adapted, we have
  \begin{equation}\label{c-15}
  \begin{split}
    &\E\Big|\sum_{l=0}^{N-1}\int_{t_l}^{t_{l+1}} \!\! \sigma(t_l, X_{t_l}^{(0)}, \law_{X_{t_l}^{(0)}})\d\big(W_r^{n_k}-W_r^{(0)}\big)
    \Big|^2\\
    &=\sum_{l=0}^{N-1}\!\E\Big[\sigma(t_l, X_{t_l}^{(0)}, \law_{X_{t_l}^{(0)}})^2\E\big[
    (W_{t_{l+1}}^{n_k}-W_{t_l}^{n_k}+W_{t_l}^{(0)} -W_{t_{l+1}}^{(0)}\big)^2\big]\Big]\\
    &\leq  \sum_{l=0}^{N-1}\!\E\Big[\sigma(t_l, X_{t_l}^{(0)}, \law_{X_{t_l}^{(0)}})^2\E\big[2(W_{t_{l+1}}^{n_k} -W_{t_{l+1}}^{(0)})^2+2(W_{t_{l }}^{n_k}- W_{t_{l }}^{(0)})^2\big]\Big].
  \end{split}
  \end{equation}
  Note that for the Brownian motion, the almost sure convergence of $W_t^{n_k}$ to $W_t^{(0)}$ implies that $\lim_{k\to \infty}\E\big[|W_t^{n_k}-W_t^{(0)}|^2\big]=0$ (cf. \cite{QG97}).
  Thus, for fixed $N\in \N$, it follows from \eqref{c-14} that
  \begin{equation}\label{c-16}
  \lim_{k\to \infty}\E\Big|\sum_{l=0}^{N-1}\int_{t_l}^{t_{l+1}} \!\! \sigma(t_l, X_{t_l}^{(0)}, \law_{X_{t_l}^{(0)}})\d\big(W_r^{n_k}-W_r^{(0)}\big)
    \Big|^2=0.
  \end{equation}
  Consequently, for any $\veps>0$, from \eqref{c-14}, we first choose an integer $N$ such that
  \[\E\Big[\int_0^T\!\!\big|\sigma(r, X_r^{(0)},\law_{X_r^{(0)}})-\sigma(\eta_{N}(r), X_{\eta_N(r)}^{(0)},\law_{X_{\eta_N(r)}^{(0)}}) \big|^2\d r\Big]\leq \frac{\veps}{12},\]
  then choose $K$ large enough such that for any $k\geq K$
  \begin{equation*}
    \E\Big|\sum_{l=0}^{N-1}\int_{t_l}^{t_{l+1}} \!\! \sigma(t_l, X_{t_l}^{(0)}, \law_{X_{t_l}^{(0)}})\d\big(W_r^{n_k}-W_r^{(0)}\big)
    \Big|^2\leq \frac{\veps}{6}.
  \end{equation*}
  Inserting the previous two estimates into \eqref{c-13}, using the arbitrariness of $\veps>0$,  we finally obtain that
  \[\lim_{k\to \infty} I\!I=0.\]
  Invoking \eqref{c-12}, we finally get
  \[\lim_{k\to \infty} \E\Big|\int_0^t\!\sigma(r,X_r^{n_k},\law_{X_r^{n_k}}) \d W_r^{n_k}-\int_0^t\!\sigma(r,X_r^{(0)},\law_{X_r^{ (0)}})\d W_r^{(0)}\Big|^2=0,\]
  and this lemma is proved.
\end{proof}
\begin{myrem}\label{rem-2}
In the argument of Lemma \ref{lem-5}, equation \eqref{c-14} uses the continuity of the paths of the process $(X_t^{(0)})_{t\in [0,T]}$. Since the optimal feedback control  $ \alpha_t^{(0)}$ may be discontinuous in $t$, our argument in Lemma \ref{lem-5} is invalid when    the diffusion coefficient $\sigma$ contains the term of control strategy $\alpha$.
\end{myrem}

\section{Dynamic programming principle and continuity of value function}

In this part we aim to establish the dynamic programming principle and study the continuity of the value function.
\subsection{Dynamic Programming Principle}

\begin{myprop}\label{lem-3}
Suppose that $\mathrm{(H1)}$ and $\mathrm{(H2)}$ hold. Then, for any $s\leq t\leq T$, $\mu\in \pb_1(\R^d)$,  it holds
\begin{equation}\label{b-1}
V(s,\mu)=\inf_{\Theta\in \Pi_{s,\mu}}\Big\{ \E_{s,\mu}\Big[\int_s^t\!f(r,X^{s,\mu,\alpha}_r, \law_{X^{s,\mu,\alpha}_r},\alpha_r)\d r+V(t, \law_{X^{s,\mu,\alpha}_t})\Big]\Big\},
\end{equation} where  $X_\cdot^{s,\mu,\alpha}$ stands for  the controlled process associated with $\Theta\in \Pi_{s,\mu}$.
\end{myprop}

\begin{proof}
  Denote by $\wt{V}(s,\mu)$ the right-hand side of \eqref{b-1}. We first prove that $V(s,\mu)\geq \wt{V}(s,\mu)$ for every $(s,\mu)\in [0,T)\times \pb_p(\R^d)$. Let $(X_t^{s,\mu,\alpha})_{t\in [s,T]}$ be the solution to \eqref{a-1}. Then,   by the flow property of SDE \eqref{a-1} (cf. \cite{BLPR}), it holds
  \[X_r^{s,\mu,\alpha}=X_r^{t, \law_{X_t^{s,\mu,\alpha}},\alpha},\quad r\in [t,T].\]
  Therefore, for any $\veps>0$, there exists a feedback control $\Theta_\veps\in \Pi_{s,\mu}$ such that
  \begin{align*}
  V(s,\mu)&\geq \E_{s,\mu}\Big[\int_s^T\!f(r, X_r^{s,\mu,\alpha}, \law_{X_r^{s,\mu,\alpha}},\alpha_r)\d r+g(X_T^{s,\mu, \alpha},\law_{X_T^{s,\mu,\alpha}}) \Big]-\veps.\\
  &\geq \E_{s,\mu}\Big[\int_s^t f(r, X_r^{s,\mu,\alpha}, \law_{X_r^{s,\mu,\alpha}},\alpha_r)\d r+\!\int_t^T f(r, X_r^{s,\mu,\alpha}, \law_{X_r^{s,\mu,\alpha}},\alpha_r)\d r\\
  &\quad \qquad \qquad  + g(X_T^{s,\mu, \alpha},\law_{X_T^{s,\mu,\alpha}})\Big]-\veps\\
  &\geq \E_{s,\mu}\Big[\int_s^t\!f(r,X_r^{s,\mu,\alpha}, \law_{X_r^{s,\mu,\alpha}},\alpha_r)\d r+V(t,\law_{X_t^{s,\mu,\alpha}})\Big]-\veps\\
  &\geq \wt{V}(s,\mu)-\veps.
  \end{align*} Letting $\veps\to 0$, we obtain that $V(s,\mu)\geq \wt{V}(s,\mu)$.

  Secondly, for any $\veps\!>0$, by the definition of $\wt{V}(s,\mu)$, there exists a   control $\Theta=(\Omega,\F,\p,$ $ \{\F_t\}_{t\geq 0}, W^\alpha_\cdot, X_\cdot^{s,\mu,\alpha}, \alpha_\cdot)\in \Pi_{s,\mu}$ such that
  \begin{equation}\label{b-2}
    \veps+\wt{V}(s,\mu) \geq \E_{s,\mu}\Big[\int_s^t f(r,X_r^{s,\mu,\alpha}, \law_{X_r^{s,\mu,\alpha}},\alpha_r)\d r+V(t,\law_{X_t^{s,\mu,\alpha}})\Big].
  \end{equation}
  Then, by the definition of $V(t,\law_{X_t^{s,\mu,\alpha}})$, there exists a   $\wt\Theta=(\wt\Omega, \tilde \F,\tilde\p, \{\tilde \F_t\}_{t\geq 0}$, $ W_\cdot^{\tilde \alpha}, X_\cdot^{\tilde \alpha},\tilde \alpha_\cdot)\in \Pi_{t, \law_{X_t^{s,\mu,\alpha}}}$ with $(\tilde\alpha_r)_{r\in [t,T]}$ representing the associated control strategy such that
  \begin{equation}\label{b-3}
  V(t,\law_{X_t^{s,\mu,\alpha}})\geq \! \E_{t,\law_{X_t^{s,\mu,\alpha}}}\!\Big[\int_t^T\!\! \!f( r, X_r^{\tilde \alpha}, \law_{X_r^{\tilde \alpha}}, \tilde \alpha_r)\d r \!+\! g\big(X_T^{\tilde \alpha}, \law_{X_T^{\tilde\alpha}}\big)\Big]\!-\!\veps .
  \end{equation}
  Notice that  since $(t,\law_{X_t^{s,\mu,\alpha}})$ is deterministic, we do not need to use the measurable selection theorem in this step.
  Denote by $R_s=\p\circ\Psi_\alpha^{-1}$ and $Q_t=\p\circ\Psi_{\tilde\alpha}^{-1}$. By virtue of \cite[Lemma 3.3]{Haus}, there exists a unique probability, denoted by $R_s\otimes_t Q_t$ on $\mathcal{Y}$ such that
  \begin{itemize}
    \item[(1)] $(R_s\otimes_t Q_t)(A)=R_s(A)$,  $\forall\,A\in \mathcal{Y}_t$.
    \item[(2)] the regular conditional probability distribution of $R_s\otimes_t Q_t$ with respect to $\mathcal{Y}_t$ is $\delta_\omega\otimes_t Q_t$, where $\delta_\omega\otimes_t Q_t$ is the unique probability measure on $\mathcal{Y}$ such that
        \begin{align*}
          &(\delta_\omega\otimes_t Q_t)\big(\{\tilde \omega\in \mathcal{Y}; \tilde \omega_r=\omega_r, \ 0\leq r\leq t\}\big)=1;\\
          &(\delta_\omega\otimes_t Q_t)(A)=Q_t(A), \quad \forall\, A\in \wt{\mathcal{Y}}^t:=\sigma\{\omega_r; \,r\in [t, T]\}.
        \end{align*}
  \end{itemize}
  Combining this fact with \eqref{b-2}, \eqref{b-3}, we obtain that
  \begin{equation*}
    \begin{split}
      \veps+\wt{V}(s,\mu)&\geq \E_{R_s\otimes Q_t}\Big[ \int_s^T\!f(r, X_r, \law_{X_r},\alpha_r)\d r+ g(X_T, \law_{X_T})\Big]-\veps\\
      &\geq V(s,\mu)-\veps.
    \end{split}
  \end{equation*}
  Letting $\veps\to 0$, we finally get $\wt{V}(s,\mu)\geq V(s,\mu)$, and hence $\wt{V}(s,\mu)=V(s,\mu)$. This completes the proof.
\end{proof}


\subsection{Continuity of value function}
  We proceed to investigate the regularity of the value function $V$.

\begin{mylem}\label{lem-1}
Assume that $\mathrm{(H1)}$, $\mathrm{(H2)}$ hold and $\sigma(t,x,\mu)$ depends only on $t,x$. For any $\pb(U)$-valued $\F_t$-adapted process $(\alpha_t)_{t\in [s,T]}$, and random variables $\xi,\tilde \xi$ with $\E\big(|\xi|+|\tilde \xi|\big)<\infty$, consider the SDEs
\begin{align*}
  \d X_t&=b(t,X_t,\law_{X_t},\alpha_t)\d t+\sigma(t,X_t)\d W_t,\quad X_s=\xi,\\
  \d \wt X_t&=b(t,\wt{X}_t,\law_{\wt{X}_t},\alpha_t)\d t+\sigma(t,\wt{X}_t)\d W_t,\quad \wt{X}_s=\tilde \xi.
\end{align*} Then the following assertions hold:
\begin{itemize}
  \item[$\mathrm{(i)}$] For $0\leq s<T$,
  \begin{equation}\label{a-6}
  \E\Big[\sup_{s\leq t\leq T} |X_t| \Big]<\infty.
  \end{equation}
  \item[$\mathrm{(ii)}$] There exists a constant $K_5>0$ such that
  \begin{equation}\label{a-7}
   \sup_{s\leq r\leq t}\W_1(\law_{X_r},\law_{\wt{X}_r}) \leq \E\sup_{s\le r\leq t}|X_r-\wt{X}_r|\leq\big(\E|\xi-\tilde \xi| \big) \e^{K_5(t-s)},\quad t\in [s,T].
  \end{equation}
\end{itemize}
\end{mylem}

\begin{proof}
  By (H2) and Burkh\"older-Davis-Gundy's inequality,
  \begin{align*}
    \E\big[\sup_{s\leq r\leq t}|X_r|\big]&\leq \E|X_s|\!+\!\E\Big[\int_s^t\! |b(r,X_r, \law_{X_r},\alpha_r)|\d r\Big] \!+\!\E\Big[\sup_{s\leq u\leq t}\Big|\int_s^u\!\!\sigma(r, X_r)\d W_r\Big|\Big]\\
    &\leq \E|\xi|\!+\!K_2\!\int_s^t\! (1\!+ \! 2\E|X_r|)\d r\!+\!C_1\E\Big[\Big(\int_s^t \!2K_2^2(1\!+\!|X_r|^2)\d r \Big)^{\frac 12}\Big]\\
    &\leq \E|\xi|\!+\!K_2(T\!-\!s) \! +\!\frac 14\E\big[\sup_{s\leq r\leq t}\!|X_r|\big]\!+\!C\!\int_s^t\!\E|X_r|\d r.
  \end{align*}
  It follows from Gronwall's inequality that
  \[\E\big[\sup_{s\leq r\leq t} |X_r|\big]\leq (\E|\xi|+C(T))\e^{C(T)(t-s)}\]
  for some constant $C(T)$ depending on $T$, which yields \eqref{a-6}.

  Applying (H1) and Burkh\"older-Davis-Gundy's inequality,
  we get
  \begin{align*}
    &\E\big[\sup_{s\leq r\leq t}|\wt{X}_r-X_r|\big]\\
    &\leq \E|\wt{X}_s\!-\!X_s|\!+\!2 K_1\!\int_s^t\!\E|\wt{X}_r\!-\!X_r|\d r\!+\!\E\Big[\sup_{s\leq u\leq t}\!\Big|\!\int_s^u\!\!\big(\sigma(r,\wt X_r)\!-\!\sigma(r,X_r)\big)\d W_r\Big|\Big]\\
    &\leq \E|\xi-\tilde\xi|\!+\!2K_1\!\int_s^t\!\E|\wt{X}_r\!-\!X_r|\d r\!+\!C\E\Big[\Big(\sup_{s\leq r\leq t}\!|\wt{X}_r\!-\!X_r|\int_s^t|\wt{X}_r\!-\!X_r|\d r\Big)^{\frac 12}\Big]\\
    &\leq \E|\xi-\tilde\xi|\!+\!\frac 14\E\big[\sup_{s\leq r\leq t}\!|\wt{X}_r-X_r|\big]+C \int_s^t\E|\wt{X}_r-X_r|\d r.
  \end{align*}
  Hence, \eqref{a-7} follows immediately by Gronwall's inequality. This lemma is proved.
\end{proof}

\begin{myprop}\label{lem-2}
Assume that   $\mathrm{(H1)}$-$\mathrm{(H3)}$ hold and $\sigma(t,x,\mu)$ depends only on $t,x$, then the value function  $V(s,\mu)$   defined in \eqref{a-4} satisfies that there exists a constant $C>0$ such that
\begin{equation}\label{cont}
|V(s,\mu)-V(s',\mu')|\leq C\big( \sqrt{|s'-s|} +\W_1(\mu,\mu')\big)
\end{equation} for any $s,s'\in [0,T]$, $\mu,\mu'\in \pb_1(\R^d)$.
\end{myprop}

\begin{proof}
  Let $0\leq s<s'\leq T$ and $\mu,\mu'\in \pb_1(\R^d)$. For any $\veps>0$,  there exists an admissible control $\Theta^\veps=(\Omega^\veps,\F^\veps,\p^\veps, \{\F^\veps_t\}_{t\geq 0}, W^\veps_\cdot, X^\veps_\cdot, \alpha^\veps_\cdot)\in \Pi_{s',\mu'}$ such that
  \[J(s',\mu;\Theta^\veps)\leq V(s',\mu)+\veps.\]
  Here $X^\veps_\cdot$ satisfies the following SDE
  \[\d X_t^\veps=b(t, X_t^\veps, \law_{X_t^\veps}, \alpha_t^\veps)\d t+\sigma(t,X_t^\veps) \d W_t^\veps,\quad X_{s'}^\veps=\xi', \ \law_{\xi'}=\mu'\]
  for $t\in [s', T]$.
  Note that $\alpha^\veps_t$ is given for $t\in [s',T]$. Define
  \[\tilde \alpha_t=\alpha\in \pb(U),\ \text{if $t\in [s,s')$};\ \ \  \tilde \alpha_t=\alpha_t^\veps,\ \text{ if $t\in [s',T]$}.\]
  Let $(\wt{X}_t)_{t\in [s,T]}$ be the unique solution to the following SDE
  \[\d \wt{X}_t=b(t,\wt{X}_t,\law_{\wt{X}_t},\tilde\alpha_t)\d t+\sigma(t,\wt{X}_t)\d W_t^\veps, \quad \wt X_s=\xi, \ \law_\xi=\mu, t\in [s,T],\] whose wellposedness is guaranteed by (H1), (H2) and the construction of $\tilde\alpha_\cdot$.
  Here the random variable $\xi$ is chosen so that
  \[\E|\xi-\xi'|=\W_1(\mu,\mu'),\]
  whose existence is a result on the existence of optimal coupling of $\mu$ and $\mu'$; see, e.g. \cite{Vil}.
  Introduce a new  term $\wt{\Theta}=(\Omega^\veps,\F^\veps,\p^\veps, \{\F^\veps_t\}_{t\geq 0},  W^\veps_\cdot, \wt{X}_\cdot, \tilde{\alpha}_\cdot)$, and under condition (H1), (H2), we can check directly that $\wt{\Theta}\in \Pi_{s,\mu}$  according to Definition \ref{def-1}.
  Therefore, by (H3),
  \begin{equation}\label{val-1}
  \begin{aligned}
V(s,\mu)-V(s',\mu')&\leq \!\E\Big[\!\int_s^T \!\!\! f(r,\wt{X}_r,\law_{\wt{X}_r},\tilde\alpha_r)\d r\!+\!g(\wt{X}_T,\law_{\wt{X}_T})\Big]\\
    &\quad - \!\E\Big[\!\int_{s'}^T \!\!\! f(r,X^\veps_r,\law_{X^\veps_r},\alpha^\veps_r)\d r \!+\!g(X_T^\veps, \law_{X_T^\veps})\Big]\!+ \veps\\
    &\leq K_3\E\Big[\!\int_{s'}^T \!\!\big(|\wt{X}_r\!-\!X_r^\veps| \!+\! \W_1(\law_{\wt{X}_r},\law_{X_r^\veps})\big)\d r\Big]\! +\! 2K_3\E|\wt{X}_T -X_T^\veps |\\
    &\qquad +K_4\E\Big[\int_s^{s'} (1+|\wt{X}_r|+\E|\wt{X}_r|)\d r \Big]\!+\veps.
  \end{aligned}
  \end{equation}
By Lemma \ref{lem-1}, (H1), (H2) and Burkh\"older-Davis-Gundy's inequality,  there exists a constant  $c_1$ such that
\begin{align*}
\E|\wt{X}_{s'}\!-\!\wt{X}_s|&\leq \int_s^{s'}\!\!K_2\big(   1+2\E|\wt{X}_r| \big) \d r \!+\!C\E\Big[\Big(\int_s^{s'}\!\! \|\sigma(r,\wt{X}_r)\|^2\d r\Big)^{\frac 12}\Big]\\
&\leq K_2\int_s^{s'}(1+2\E|\wt{X}_r|)\d r\!+\! C\E\Big[K_2\sup_{s\leq r\leq s'}\! (1+|\wt{X}_r|)\sqrt{|s'-s|}\Big]\\
&\leq c_1(|s'-s|+\sqrt{|s'-s|}).
\end{align*}

Furthermore, by  Lemma \ref{lem-1}, there exist constants $c_2,\,c_3>0$ such that
\begin{align*} V(s,\mu)-V(s',\mu')&\leq c_2\big( \E|\wt{X}_{s'}-\xi'|+|s'-s|\big)+\veps\\
&\leq c_2\big(\E|\wt{X}_{s'}-\wt{X}_s|+\E|\xi-\xi'|+ |s'-s|\big)+\veps\\
&\leq c_3\big( \sqrt{|s'-s|}+\W_1(\mu,\mu')\big)+\veps.
\end{align*} Letting $\veps\to 0$, we get
\[V(s,\mu)-V(s',\mu')\leq c_3\big( \sqrt{|s'-s|}+\W_1(\mu,\mu')\big).\]
Similarly, we can prove that $V(s',\mu')-V(s,\mu)\leq c_3( \sqrt{|s'-s|}+\W_1(\mu,\mu'))$, and obtain the desired conclusion.
\end{proof}



\end{document}